\documentclass[11pt]{amsart}
\usepackage{geometry}                
\usepackage{graphicx}
\usepackage{amssymb}
\usepackage{epstopdf}
\theoremstyle{plain}
\newtheorem{theo}{Theorem}[section]
\newtheorem{lem}[theo]{Lemma}
\newtheorem{cor}[theo]{Corollary}
\newtheorem{prop}[theo]{Proposition}

\theoremstyle{definition}

\def\e{\varepsilon}

\def\a{\alpha}
\def\b{\beta}
\def\P{\mathfrak P}
\def\th{\theta}
\def\d{\delta}
\def\vt{\vartheta}
\def\D{\Delta}
\def\G{\Gamma}
\def\g{\gamma}

\def\DD{\mathcal D}
\def\l{\lambda}
\def\L{\Lambda}
\def\EE{\mathcal E}

\def \R{\mathbb R}
\def \N{{\mathbb N}}

\def \Z{\mathbb Z}

\def \H{\mathbb H}
\def \F{\mathbb F}
\def \P{\mathbb P}

\def\om{\omega}
\def\Om{\Omega}

\def\C{\mathbb C}
\def\CC{\mathcal C}

\def\O{\mathcal O}

\def\({\biggl(}
\def\){\biggr)}

\def\<{\bold\langle}
\def\>{\bold\rangle}

\def\LL{{{\mathcal L}}}

\def\M{\mathcal {P}}

\begin{document}
\title{Analyticity of the entropy for some random walks}
\author{Fran\c cois Ledrappier}
\address{LPMA, UMR CNRS 7599, Universit\'e Paris 6, Bo\^ite Courrier 188, 4, Place Jussieu, 75252 PARIS cedex 05, France}\email{francois.ledrappier@upmc.fr}


\maketitle
\begin{abstract}
We consider non-degenerate, finitely supported random walks on  a free group. We show that the entropy and the linear drift vary analytically with the probability of constant support.
\end{abstract}
\section{Introduction}

Let $F$ be a finitely generated group  and for $x \in F$, denote $|x|$ the word length of $x$. Let $p$ be a finitely supported probability measure on $F$ and define inductively,
with $ p^{(0)} $ being the Dirac measure at the identity $e$,
$$ p^{(n)} (x) \; = \; [p^{(n-1)} \star p] (x) \; = \;  \sum _{y \in F} p^{(n-1)} (xy^{-1}) p(y).$$
Some of the asymptotic properties of the probabilities $p^{(n)}$ as $n \to \infty $  are reflected in  two nonnegative numbers, the entropy $h_p$ and the linear drift $\ell _p$:
$$ h_p := \lim_n -\frac{1}{n} \sum _{x\in F} p^{(n)}(x) \ln p^{(n)}(x),  \quad \ell _p := \lim _n \frac{1}{n} \sum _{x \in F} |x| p^{(n)}(x).$$
Erschler asks whether  $h_p$ and $\ell _p$ depend continuously on $p$ (\cite{E}). In this note, we fix a finite set $B \subset F$ such that $\cup _n B^n = F$ and we consider probability measures in $\M(B)$, where $\M(B)$ is the set of probability measures $p$ such that $p(x) >0$ if, and only if, $x \in B$. The set $\M(B)$ is naturally identified with an open subset of the probabilities on $B$ which is  an open bounded convex domain in $\R^{|B|- 1}$. We show: \begin{theo}\label{main} Assume $F = \F_d$ is the free group with $d$ generators, $B$ is a finite subset of $F$  such that $\cup _n B^n = F$.  Then, with the above notation, the functions $p \mapsto h_p$ and $p \mapsto \ell _p$ are real analytic on $\M(B)$. \end{theo}

Continuity of the entropy and of the linear drift is known for probabilities with first moment on a Gromov-hyperbolic group (\cite{EK}). Also in the case when $B$ is a set of free generators, there are formulas for the entropy  and the linear drift which show that they are real analytic functions of the directing probability (see \cite{De2} or imbed \cite{DM} in the formulas (\ref{entropy}) and (\ref{drift}) below). Similar formulas have been found for braid groups (\cite{M}) and free products of finite groups or graphs (\cite{MM}, \cite{G1}, \cite{G2}), but as soon as the set $B$ is not reduced to the natural generating set, there is no direct formula for $h_p $ or $\ell _p $ in terms of $p$.

The ratio $h_p/ \ell_p $ has a geometric interpretation as the Hausdorff dimension $D_p$ of the unique stationary measure for the action of $F$ on the space $\partial F$ of infinite reduced words.  It follows from Theorem \ref{main} that this dimension $D_p$ is also real analytic  in $p$, see below Corollary \ref{dim} for a  more precise statement.   Ruelle (\cite{R3}) proved that the Hausdorff dimension of the Julia  set of a rational function, as long as it is hyperbolic,  depends real analytically of the parameters and our approach is inspired by \cite{R3}. We first  review properties of the random walk on $F$ directed by a probability $p$. In particular, we can  express $h_p $ and $\ell _p$  in terms of the  exit measure $p^ \infty $ of the random walk on the boundary $\partial F$ (see \cite{Le} and section \ref{rw} for background and notation). We then express this exit measure using thermodynamical formalism: if one views $\partial F$ as a one-sided subshift of finite type, the exit measure $p^\infty $ is the isolated eigenvector of maximal eigenvalue for a dual transfer operator $\LL _p^\ast$ involving the Martin kernel of the random walk. Finally, from the description of the Martin kernel by Derriennic (\cite{De1}), we prove that the mapping $p \mapsto \LL_p $ is real analytic. The proof uses contractions in projective metric  on complex cones (\cite{Ru}, \cite {Du}) and I  want to thank Lo\"ic Dubois for useful comments. Regularity  of $p \mapsto p^\infty $  and Theorem \ref{main} follow. 

\ 

Our argument may apply to other similar settings. For instance, let $\pi : \F_d \to SO(k,1)$ be a faithful Schottky representation of the free group $\F_d$ as a convex cocompact group of $SO(k,1)$. Namely, $SO(k,1) $ is considered as a group of isometries of the hyperbolic space $\H^k$ and there are $2d$ disjoint open halfspaces $H_a$ associated to the generators and their inverses in such a way that $\pi (a) $ sends the complement of $H_{a^{-1}}$ onto the closure of $H_a$ in $\H^k$.
Then, another natural asymptotic quantity is the Lyapunov exponent $\g_p$:
$$
\g_p \; := \;  \lim_n \frac{1}{n} \sum _{x\in F}  p^{(n)}(x) \ln \|\pi (x) \|,
$$
where $\|\cdot \| $ is  some norm on matrices. Then,
\begin{theo}\label{expo} Assume $\F_d$ is represented as a convex cocompact subgroup of $SO(k,1)$ as above, and $B$ is  a finite subset $B \subset F$ such that $\cup _n B^n = F$. Then the function $p \mapsto \g_p $ is a real analytic function on $\M(B)$.
\end{theo}
Analyticity of the exponent of an independent random product of matrices is known for {\it { positive }} matrices  (\cite{R2}, \cite{P}, \cite{H}). Here we show it for matrices in some {\it {discrete }} subgroup. It is possible that our approach yield similar results for more general discrete subgroups of $SO(k,1)$ or even for all Gromov-hyperbolic groups. 

 \

In  the note, the letter $C$ stands for a real number independent of the other variables, but which may vary from line to line. In the same way, the letter  $ \O_p$ stands for a neighborhood of  $p \in \M(B)$ in $\C^B$ which may vary from line to line. 

\section{Convolutions of $p$}\label{rw}

We recall in this section the properties of the convolutions $p^{(n)}$ of a finitely supported probability measure $p$ on the free group  $\F_d = F$. We follow the notation from \cite{Le}. Any element of $F$ has a unique reduced word representation in generators $\{a_1, \cdots, a_d, a_{-1}, \cdots,  a_{-d} \}$. Set $\d (x,x) = 0 $ and, for $x \not = x'$, $\d (x,x') = \exp -(x\wedge x')$, where $(x\wedge x')$ is the number of common letters at the beginning of the reduced word representations of $x$ and $x'$. Then 
$\d$ defines a metric on $F$ and extends to the completion $F \cup \partial F$ with respect to $\d$. The boundary $\partial F$ is a compact space which can be represented as the space of infinite reduced words. Then the distance between two distinct infinite reduced words $\xi $ and $\xi'$ is given by 
$$ \d (\xi, \xi' ) \;= \; \exp - (\xi  \wedge \xi '),$$
where $(\xi \wedge \xi') $ is the length of the initial common part of $\xi $ and $\xi '$.

There is a natural continuous action of $F$ over $\partial F$ which extends the left action of $F$ on itself: one concatenates the reduced word representation of $x \in F$ at the beginning of the infinite word $\xi$ and one obtains a reduced word by making  the necessary reductions.  A probability measure $\mu $ on $\partial F$ is called stationary if it satisfies $$ \mu \; = \; \sum _{x \in F} p(x) x_\ast \mu .$$
There is a unique stationary probability measure on $\partial F$, denoted $p^\infty $ and the entropy $h_p$ and the linear drift $\ell _p$ are given by the following formulae:
\begin{equation}\label{entropy}
h_p \; = \; - \sum _{x \in F} \left( \int _{\partial F}\ln  \frac {dx^{-1}_\ast p^\infty }{dp^\infty } (\xi) dp^\infty (\xi)\right) p(x) ,\end{equation}
\begin{equation}\label{drift}
\ell_p \; = \; \sum _{x \in F} \left( \int _{\partial F}\th _\xi (x^{-1}) dp^\infty (\xi)\right) p(x) ,\end{equation}
where $\th_\xi (x) = |x| - 2(\xi \wedge x) = \lim _{y \to \xi} (|x^{-1}y | - |y|) $ is the Busemann function.

\

Observe that in both expressions, the sum is a finite sum over $ x \in B$. In the case of a  finitely supported random walk on a general group, formula (\ref{entropy})
holds, but with $(\partial F, p^\infty )$ replaced by the Poisson boundary of the random walk (see \cite{F}, \cite {K}); formula (\ref{drift}) also 
holds, but with $(\partial F, p^\infty )$ replaced by some stationary measure on the Busemann boundary of the group (\cite{KL}). 

\

Recall that in the case of the free group the Hausdorff dimension of the measure $p^\infty $ on $(\partial F, \d)$ is given by $h_p / \ell _p$ (\cite{Le}, Theorem 4.15). So we have the following Corollary of Theorem \ref{main}:
\begin{cor}\label{dim} Assume $F = \F_d$ is the free group with $d$ generators, $B$ is a finite subset of $F$  such that $\cup _n B^n = F$.  Then, with the above notation, the Hausdorff dimension of the stationary measure on $(\partial F, \d ) $ is a real analytic  function of  $p$ in $\M(B)$. \end{cor}

\

The Green function $G(x)$ associated to $(F,p) $ is defined by $$ G(x) \;  = \; \sum _{n=0}^\infty p^{(n)}(x)$$ 
(see Proposition \ref{decay} below for the convergence of the series).
For $y \in F$, the Martin kernel $K_y$ is defined by $$ K_y (x) = \frac {G(x^{-1}y)}{G(y)}.$$
Derriennic (\cite{De1}) showed that $y_n \to \xi \in \partial F$ if, and only if, the Martin kernels $K_{y_n} $ converge towards a function $K_\xi $ called the Martin kernel at $\xi$. We have (see e.g. \cite{Le} (3.11)): 
\begin{equation*}  \frac{dx_\ast p^\infty }{dp^\infty }(\xi) \; = \; K_\xi (x). \end{equation*}

\section{Random walk on $F$}\label{rw2}

The quantities introduced in Section \ref{rw} can be associated with the trajectories of a random walk on $F$. In this section, we recall the corresponding notation and properties. Let $\Om = F^\N $ be the space of sequences of elements of $F$, $M$ the product probability $p^\N $. The random walk is described by the probability $\P$
on the space of paths $\Om $, the image of $M$ by the mapping:
$$ (\om _n )_{n \in \Z} \mapsto (X_n )_{n \geq 0 }, {\textrm{ where }} X_0 = e {\textrm{ and }} X_n = X_{n-1}\om _n {\textrm{ for }} n>0.$$
In particular, the distribution of $X_n$ is the convolution $p^{(n)}$. The notation $p^\infty $ reflects the following 
\begin{theo}[Furstenberg, see \cite{Le}, Theorem 1.12] There is a mapping $X_\infty : \Om \to \partial F$ such that for $M$-a.e. $\om $, 
$$ \lim _n X_n  (\om ) \; = \; X_\infty (\om ).$$
The image measure $p^\infty $ is the only stationary probability measure on $\partial F$.
\end{theo}

\

For $x,y \in F$, let $u(x,y)$ be  the  probability  of eventually reaching $y$ when starting from $x$. By left invariance, $u(x,y ) = u (e, x^{-1}y) $. Moreover, by the strong Markov property, $G(x) = u(e,x) G(e)$ so that we have:
\begin{equation}\label{Martin} K_y(x) = \frac {u(x,y)}{u(e,y)}. \end{equation}
By definition, we have $0< u(x,y) \leq 1$.  The number $u(x,y) $ is given by the sum of the probabilities of the paths going from $x$ to $y$ which do not visit $y$ before arriving at $y$. 
 \begin{prop}\label{decay} Let $p \in \M(B)$. There are  numbers $C$ and  $\zeta, 0 < \zeta < 1$ and a  neighborhood $\O_p$ of $p$ in $\C^B$ such that for all $q \in \O_p$, all $x \in F$ and all $n \geq 0$, 
$$| q|^{(n)} (x)\leq C \zeta ^n.$$ \end{prop}
\begin{proof} Let $q \in \C^B$. Consider the convolution operator $P_q$ in $\ell _2 (F, \R)$ defined by:
$$ P_qf(x) = \sum _{y\in F} f(xy^{-1}) |q|(y).$$
 Derriennic  and Guivarc'h (\cite {DG}) showed that for $p \in \M(B)$,  $P_p$ has spectral radius  smaller than one. In particular, there exists $n_0$ such that the operator norm of $P_p^{n_0} $ in $\ell _2(F)$ is smaller than one. Since $B$ and $B^{n_0}$ are finite, there is a neighborhood $\O_p$ of $p$ in $\C^B$ such that for all $q \in \O_p$, $\| P_q^{n_0} \|_2 < \l $ for some $\l <1$ and $\| P_q^k \|_2 \leq C$ for $1\leq k \leq n_0.$
 It follows that for all $q \in \O_p $, all $n \geq 0$,
 $$ \|P_q^n \|_2 \; \leq \; C\l ^{[n/n_0]} .$$
In particular, for all $x \in F$, $|q|^{(n)} (x) = |[P_q^n \d_e](x)|\leq |P_q^n \d _e|_2  \leq  C \l ^{[n/n_0]} | \d_e |_2 \leq  C\l ^{[n/n_0]}.$
\end{proof}

\

Fix $p \in \M(B)$. For $x\in F$, $V$ a finite subset of $F$ and $v\in V$, let $\a_x^V(v)$ be the probability that the first visit in $V$ of the random walk starting from $x$ occurs  at $v$. We have $0 < \sum _{v\in V} \a_x^V (v) \leq 1$ and
\begin{prop}\label{cont} Fix $x,V$ and $v$. The mapping $p \mapsto \a_x^V(v) $ extends to an analytic function on a neighborhood of $\M (B)$ in $\C^B$. \end{prop}
\begin{proof} The number $\a_x^V(v)$ can be written as the sum of the probabilities $\a _x ^{n,V} (v) $ of entering  $V$ at $v$ in exactly $n$ steps. The function $p \mapsto \a _x ^{n,V} (v) $ is a  polynomial of degree $n$ on $\M (B)$:
$$  \a _x ^{n,V} (v)  \; = \; \sum _\EE q_{i_1}q_{i_2}\cdots q_{i_n}, $$
where $\EE $ is the set of paths $\{ x, xi_1, xi_1i_2, \cdots , xi_1i_2 \cdots i_n = v \}$ of length $n$ made of steps in $B$ which start from $x$ and enter $V$ in $v$.  By Proposition \ref{decay}, there is a neighbourhood $\O_p$ of $p$ in $\M(B)$ and numbers $C$, $\zeta, 0 < \zeta < 1$, such that for $q \in \O_p$ and for all $y \in F$,
$$ |q|^{(n)} (y)  \; \leq \; C \zeta ^n . $$
It follows that for $q \in \O_p$, $$ | \a _x ^{n,V} (v) |  \leq  C |q|^{(n)} (x^{-1}v) \leq   C \zeta ^n.$$
Therefore,  $q \mapsto \a_x^V(v)$ is given locally by a uniformly converging series of polynomials, it is an analytic  function on $\O := \cup _p \O_p$. \end{proof}

\

\section{Barriers and H\"older property of the Martin kernel}

Set  $r = \max \{ |x|; x \in B\}$. A set $V$ is called a barrier between $x$ and $y$ if $\d (x,y) > r$ and if there exist two points $z $ and $z'$ of the geodesic between $x$ and $y$, distinct from $x$ and $y$ such that $\d (z,z') = r-1 $ and $V$ is the intersection of the two balls of radius $r-1$ centered at $z$ and at $z'$. The basic geometric Lemma is the following:
\begin{lem}[\cite{De1}, Lemme 1]\label{barrier} If $x$ and $y$ admit a barrier $V$, then every trajectory of the random walk starting from $x$ and reaching $y$ has to visit $V$ before arriving at $y$. \end{lem}

For $V,W$ finite subsets of $F$, denote by $A_V^W$ the matrix such that the row vectors are the $\a_v^W (w), w \in W$. In particular, if $W = \{y\}$, set $u_V^y $ equal to  the (column) vector $$ u_V^y = A_V^{\{y\}} = (\a _v^{\{y\}} (y))_{v\in V} = (u(v,y))_{v\in V}.$$

With this notation,  Lemma \ref{barrier} and the strong Markov property yield, if $x$ and $y$ admit $V$ as a barrier:
$$ u(x,y) \; = \; \sum \a_x^V (v) u(v,y) \; = \; \< \a_x^V, u_V^y \>, $$
with the natural scalar product on $\R ^{V}$. Then, Derriennic makes  two observations: firstly, this formula iterates when one has $k$ successive disjoint barriers between $x$ and $y$
and secondly there are only a finite number of possible  matrices $A_V^W$ when $V$ and $W$ are successive disjoint barriers with $\d (V,W) = 1$. This gives the following formula for $u(x,y)$:
\begin{lem}[\cite{De1}, Lemme 2]\label{basiclemma} Let $p \in \M(B) $. There are $N$ square matrices with the same dimension $A_1, \cdots, A_N$, depending on $p$, such that for any $x,y \in F$, if $V_1, V_2, \cdots, V_k$ are disjoint successive barriers between $x$ and $y$ such that $\d (V_i, V_{i+1} ) = 1$ for $i = 1, \cdots, k-1$, there are $(k-1) $ indices $j_1, \cdots j_{k-1}$, depending only on the sequence $V_i$ such that:
\begin{equation}\label{basic} u(x,y) = \< \a_x^{V_1}, A_{j_1}\cdots A_{j_{k-1}} u_{V_k}^y \>. \end{equation}\end{lem}

By construction, the matrices $A_j$ have nonnegative entries and satisfy $\sum _w A_j (v,w) \leq 1$. Moreover, we have the following properties:
\begin{prop}[\cite{De1}, Corollaire 1] \label{zeros} Assume the set $B$ contains the generators and their inverses, then for each $p \in \M(B)$, for each $j = 1, \cdots , N$, the matrix $A_j $ has all its $0$ entries in full columns. \end{prop}
From the proof of proposition \ref{zeros}, if the set $B$ contains the generators and their inverses and $A_j = A_{V_j}^{V_{j+1}}$, columns of $0$'s correspond to the subset $W_{j+1}$ of points in $V_{j+1}$ which cannot be entry points from paths starting in $V_j$. In particular, they depend only of the geometry of $B$ and are the same for all $p \in \M(B)$.

We may -- and we shall from now on -- assume that the set $B$ contains the generators and their inverses. Indeed, since $h_{p^{(k)}} = k h_p$ and $\ell _{p^{(k)}} = k \ell _p $, we can replace in Theorem \ref{main} the probability $p$ by a convolution of order high enough that the generators and their inverses have positive probability. Then, by Proposition \ref{zeros} the matrices $A_j (q)$ have the same columns of zeros for all $q \in \M(B)$. Moreover, 

\begin{prop}\label{cont2} For each $j = 1, \cdots, N$, the mapping $p \mapsto A_j$ extends to analytic function on a neighborhood of $\M (B)$ in $\C^B$ into the set of complex matrices with the same configuration of zeros as $A_j$. \end{prop}
\begin{proof} The proof is completely analogous to the proof of Proposition \ref{cont}; one may have to take a smaller neighborhood for the sake of  avoiding  introducing new zeros. \end{proof}

\ 

We are interested in the function $\Phi : \partial F \to \R, \Phi (\xi ) = - \ln K_\xi (\xi _1).$ By (\ref{Martin}), (\ref{basic}) and Deriennic's Theorem, we have:
\begin{eqnarray*}
 \Phi (\xi ) &=& -\ln \lim_{ n \to \infty } K_{\xi _1\xi_2 \cdots \xi _n} (\xi _1) \\
&=& -\ln \lim _{n \to \infty } \frac{u(\xi _1,\xi _1\xi _2 \cdots \xi _n)}{u(e ,\xi _1\xi _2 \cdots \xi _n)} \\
&=& -\ln \lim _{k \to \infty } \frac{ \< \a_{\xi _1}^{V_1(\xi ) }, A_{j_1}(\xi )\cdots A_{j_{k-1}}(\xi )  u_{V_k(\xi )}^{y_k} \> } { \< \a_{e}^{V_1(\xi ) }, A_{j_1}(\xi )\cdots A_{j_{k-1}}(\xi )  u_{V_k(\xi ) }^{y_k} \> }, \end{eqnarray*}
where $A_{j_s}(\xi ) = A_{V_s(\xi)}^{V_{s+1}(\xi)} $, the $V_s(\xi )$ are successive disjoint barriers between $\xi _1$ and $\xi $ with $\d (V_s (\xi ), V_{s+1}(\xi)) = 1$ for all $s > 1$, $\d(\xi _1, V_1) = 1$  and $y_k $ is the closest point beyond  $V_k$ on the geodesic from $\xi _1$ to $\xi $.

\

Define on the nonnegative  convex cone $C_0$ in $\R^m$ the projective distance between half lines as $$ \vt (f,g) \; := \; | \ln [f,g,h,h']| ,$$
where $h,h'$ are the intersections of the boundaries of the cone with the plane $(f,g)$ and $[f,g, h, h'] $ is the cross ratio of the four directions in the same plane. Represent the space of directions as the sector of the unit sphere $ D = C_0 \cup S^{m-1} $; then, $\vt $ defines a distance on $D$. Let $A$ be a $m \times m $ matrix with nonpositive entries and let $T : D \to D$ be the  projective action of $A$. Then, by \cite{B}:
\begin{equation}\label{contr}
\vt (Tf,Tg) \; \leq \b \vt (f,g), {\textrm { where }} \b = \tanh \left( \frac{1}{4} {\textrm { Diam }} T(D) \right). \end{equation}
When $A_j $ is one of the matrices of Lemma \ref{basiclemma}, it acts on $\R ^{V}$ and the image  $T_j (D)$ has finite diameter, so that  $ \b_j := \tanh \left( \frac{1}{4} {\textrm { Diam }} T_j(D) \right) < 1$.  Set $ \b_0 := \max _{j = 1, \cdots, N}  \b_j.$ Then, $\b_0 < 1$.

Set $\displaystyle f_k (\xi ) := \frac{u_{V_k(\xi)}^{y_k}}{\| u_{V_k(\xi)}^{y_k}\|}, \a(\xi) := \a_{e}^{V_1(\xi )}, \a_1(\xi) := \a_{\xi _1}^{V_1(\xi )}$.  For all $\xi $, $f_k (\xi ) \in D$ and $\a(\xi), \a_1(\xi) \in C_0 -\{0\} $. The above formula for $\Phi (\xi) $ becomes:
\begin{equation}\label{Derriennic} 
\Phi (\xi ) =  -\ln \lim _{k \to \infty } \frac{ \< \a_1(\xi), T_{j_1}(\xi )\cdots T_{j_{k-1}}(\xi )  f_k(\xi)  \> } { \< \a(\xi), T_{j_1}(\xi )\cdots T_{j_{k-1}}(\xi )  f_k(\xi ) \> }.
\end{equation}

\

\begin{prop}\label{holder} Fix $p \in \M$. The function $\xi \mapsto \Phi (\xi)  $ is H\"older continuous on $\partial F$. \end{prop}
\begin{proof} Let $\xi, \xi' $ be two points of $\partial F$ with $\d(\xi , \xi ' ) \leq \exp (-((n+1) r +1))$. The points $\xi $ and $\xi '$ have the same first $(n+1)r +1$ coordinates. In particular, $V_s (\xi ) = V_{s} (\xi ') $ for $1 \leq s \leq n$. By using (\ref{Derriennic}), we see that $\Phi (\xi ') - \Phi (\xi )$ is given by the limit, as $k$ goes to infinity, of 
$$ \ln \frac{ \< \a_1(\xi ) , T_{j_1}(\xi )\cdots T_{j_{k-1}}(\xi )  f_k(\xi)  \> }{ \< \a_1(\xi ') , T_{j_1}(\xi ')\cdots T_{j_{k-1}}(\xi ')  f_k(\xi')  \> }
\frac{ \< \a(\xi ') , T_{j_1}(\xi ')\cdots T_{j_{k-1}}(\xi ')  f_k(\xi ') \> }{ \< \a(\xi) , T_{j_1}(\xi )\cdots T_{j_{k-1}}(\xi )  f_k(\xi ) \> }.$$
We have $ \a_1(\xi) =  \a_1(\xi ') =: \a_1,  \a(\xi)  =  \a(\xi ') =: \a $ and $T_{j_s} (\xi ) = T_{j_s} (\xi ') =: T_{j_s} $ for $ s= 1, \cdots, n.$  Moreover, for any $f,f' \in D$, 
$$  \vt \left( T_{j_1}(\xi )\cdots T_{j_{k-1}}(\xi )  f, T_{j_1}(\xi ')\cdots T_{j_{k-1}}(\xi ')  f' \right) \;= \;  \vt \left(T_{j_1}\cdots T_{j_{n-1}}  g_k, T_{j_1}\cdots T_{j_{n-1}} g' _k\right) $$
for $g_k = T_{j_n} T_{j_{n+1}}(\xi)\cdots T_{j_{k-1}}(\xi )  f, g' _k = T_{j_n} T_{j_{n+1}}(\xi')\cdots T_{j_{k-1}}(\xi ')  f'$. 

\

We have $\vt (g_k,g'_k ) \leq {\textrm{ Diam } } T_{j_n } D < \infty $ and, by repeated application of (\ref{contr}),
\begin{equation} \label{regularity} \vt \left(T_{j_1}\cdots T_{j_{n-1}}  g_k, T_{j_1}\cdots T_{j_{n-1}}  g'_k   \right) \; \leq \; \b_0^{n-1} \vt \left(g_k, g'_k\right) \; \leq C \b_0^n .\end{equation}
Using all the above notation, we get
\begin{equation}\label{holderratio} \Phi (\xi ) - \Phi (\xi ') = \ln \lim _k \frac {\< \a_1, T_{j_1}\cdots T_{j_{n-1}}  g'_k \>} {\< \a_1, T_{j_1}\cdots T_{j_{n-1}}  g_k \>} \frac  {\< \a, T_{j_1}\cdots T_{j_{n-1}}  g_k \>}  {\< \a, T_{j_1}\cdots T_{j_{n-1}}  g'_k \>}.\end{equation}

As $\xi $ varies, $\a $ and $\a_1$ belong to a finite family of vectors of $C_0 - \{0\}$. It then follows from (\ref{regularity}) that, as soon as $\d(\xi , \xi ' ) \leq \exp (-((n+1) r+1))$, $|\Phi (\xi ) - \Phi (\xi ')| \leq C \b_0^n .$
\end{proof}

\

Let us choose $\b, \b_0 ^{1/r} < \b  < 1$, and consider the space $\G_\b $ of functions $\phi $ on $\partial F$ such that there is a constant $C_\b$ with the property that, if the points $\xi $ and $\xi '$ have the same first $n $ coordinates, then $| \phi (\xi ) - \phi (\xi ')| < C_\b \b ^n. $ For $\phi \in \G_\b $, denote $\| \phi \|_\b $ the best constant $C_\b $  in this definition. The space $\G_\b $ is a Banach space for the norm $ \| \phi \|  :=   \| \phi \|_\b  + \max _{\partial F}  |\phi | .$ Proposition \ref{holder} says that for $p \in \M(B)  $, the function $\Phi _p(\xi ) = 
- \ln K_\xi (\xi _1)$ belongs to $\G_\b$.

\section{Regularity of the Martin kernel}

We want to extend the mapping $p \mapsto \Phi_p$ to a neighborhood $\O _p$ of $p$ in $\C^B$. Firstly, we redefine $\G_\g$ as the space of complex  functions $\phi $ on $\partial F$ such that there is a constant $C_\g$ with the property that, for all $n \geq 0$, if the points $\xi $ and $\xi '$ have the same first $n $ coordinates, then $| \phi (\xi ) - \phi (\xi ')| < C_\g \g ^n.$ The space $\G_\g $ is a complex Banach space for the norm $ \| \phi \|  :=   \| \phi \|_\g  + \max _{\partial F}  |\phi | ,$ where  $\| \phi \|_\g $ the best possible constant $C_\g $. In this section,   we find a neighborhood $\O_p $ and a $\g = \g (p), 0< \g <1$, such that formula (\ref{Derriennic}) makes sense on $\O_p$ and defines a function in $\G_\g$. 

\

In  recent papers, Rugh (\cite{Ru}) and Dubois (\cite{Du}) show how to extend (\ref{contr}) to the complex setting. In a complex Banach  space $X$, they define a $\C$-cone as a  subset invariant by multiplication by $\C$, different from $\{0\}$ and not containing any complex 2-dimensional subspace in its closure. A $\C$-cone $\CC$ is called linearly convex if  each point in the complement of $\CC $ is contain in a complex hyperplane not intersecting $\CC$. Let $K < +\infty $. A $\C$-cone $\CC$ is called $K$-regular if it has some interior and if, for each vector space $P$ of complex dimension 2, there is some nonzero linear form $ m \in X^\ast $ such that, for all $u  \in \CC \cap P$, 
\begin{equation*} \| m\| \|u\| \; \leq \; K | \<m,u \>|.\end{equation*}
Let $\CC $ be a linearly convex $\C$-cone. A projective distance $\vt_\CC$ on $(\CC - \{0\}) \times (\CC - \{0\})$ is defined as follows (\cite{Du}, Section 2): if $f$ and $g$ are colinear, set $\vt _\CC (x,y) = 0 $; otherwise, consider the following set  $E(f,g)$:
$$ E(f,g)\;  := \;\{ z, z \in  \C,   zf-g   \not \in \CC \},$$
and then define 
$$\vt _\CC (f,g) = \ln \frac {b}{a}, \quad {\textrm {where }} b = \sup |E(f,g)| \in (0,+ \infty], a = \inf |E(f,g)| \in [0,+\infty ).$$

\begin{prop}[\cite{Du}, Theorem 2.7] Let $X_1,X_2$ be complex Banach spaces, and let $\CC_1 \subset X_1, \CC_2 \subset X_2$ be complex cones. Let $A: X_1 \to X_2$ be a linear map with $A(\CC_1 - \{0\}) \subset (\CC_2 - \{0\})$ and assume that $$\D : = \sup _{f,g \in (\CC_1 - \{0\}) } \vt _{\CC _2}(Af, Ag ) < +\infty .$$ Then, for all $f,g \in \CC_1$,
\begin{equation}\label{ccontr} \vt _{\CC _2} (Af, Ag) \; \leq \; \tanh \Big(\frac {\D }{4}\Big) \vt _{\CC_1} (f,g). \end{equation}
\end{prop}

\begin{prop}[\cite{Du}, Lemma 2.6]\label{cdist} Let $\CC $ be a $K$-regular, linearly convex $\C$-cone and let $f\sim g $ if, and only if,  there is $\l , \l \not = 0 $ such that $ \l f = g $. Then $\vt _\CC $ defines a complete projective metric on $\CC / \sim $. Moreover, if $f,g \in \CC, \|f\|= \|g\| =1$,  then there is a complex number $\rho$ of modulus 1,  $\rho = \rho (f,g) $, such that 
\begin{equation}\label{compar} \|\rho f - g \| \; \leq \; K \vt_\CC (f,g).\end{equation}
\end{prop}

\begin{prop}[\cite{Ru}, Corollary 5.6, \cite{Du}, Remark 3.6]\label{ccone} For $m \geq 1$, the set
$$ \C _+^ m = \{ u \in \C^m: \mathfrak {Re}( u_i \overline {u_j}) \geq 0, \forall i,j\} = \{ u \in \C^m: |u_i + u_j | \geq |u_i - u_j |, \forall i,j \} $$ is a regular linearly convex  $\C$-cone. The inclusion $$\pi :(C_0 -\{0\} , \vt)\longrightarrow (\C_+^m - \{0\}, \vt_{\C_+^m} ) $$
is an isometric embedding.
\end{prop}

Moreover,  \cite{Du} studies and  characterizes the $m\times m $ matrices which preserve $\C_+^m $. We need the following properties. Let $A$ be a $m\times m $ matrix with all $0$ entries in $m'$ full columns and $\l_1, \cdots, \l_m $ the $(m-m')$-row vectors made up of the nonzeros entries of the row vectors of $A$. Set:
$$ \d_{k,l} := \vt_{\C_+^{m-m'}}  (\l_k, \l_l ), \quad \D_{k,l} := {\textrm {Diam}}_{{\textrm{RHP}}} \Big\{ \frac {\< \l_k, x \>}{\< \l_l, x \>} ; x \in (\C_+^{m-m'})^\ast, x \not = 0 \Big\},$$
where ${\textrm {Diam}}_{{\textrm{RHP}}}$ denotes the diameter with respect to the Poincar\'e metric of the right half-plane.  Observe that ${\textrm {Diam}}_{\vt_{\C_+^{m}}} (A(\C_+^m - \{0\})) = {\textrm {Diam}}_{\vt_{\C_+^{m}}} (A(\C_+^{m-m'} - \{0\})).$ Then, we have (\cite{Du}, Proposition 3.5):
\begin{equation}\label{diam1}
{\textrm {Diam}}_{\vt_{\C_+^{m}}} (A(\C_+^m - \{0\})) \; \leq \; \max_{k,l} \d_{k,l} + 2 \max _{k,l} \D_{k,l} \;\leq \; 3  {\textrm {Diam}}_{\vt_{\C_+^{m}}} (A(\C_+^m - \{0\})) .
\end{equation}
From the proof of Proposition 3.5 in \cite{Du}, in particular from equation (3.12), it also follows that for a real matrix $A$:
$${\textrm {Diam}}_{\vt_{\C_+^{m}}} (A(\C_+^m - \{0\}))  \;\leq \; 3  {\textrm {Diam}}_\vt (A(\R_+^m - \{0\})) .$$
Fix $p \in \M(B)$. We choose $\g= \g(p) <1 $ such that $$ 9 (\tanh)^{-1} \b_0 < (\tanh)^{-1} (\g ^{2r}). $$
Then for the  real matrices $A = A_1(p), \cdots, A_N(p) $,
\begin{equation}\label{diam2} 
 3 {\textrm {Diam}}_{\vt_{\C_+^{m}}} (A(\C_+^m - \{0\}))  \leq  9 {\textrm {Diam }}_\vt (A(\R_+^m - \{0\}))  \leq 36  (\tanh)^{-1} \b_0 < 4 (\tanh)^{-1} (\g ^{2r}).
 \end{equation}

\begin{prop}\label{Lip} Fix $p \in \M(B).$ There is a neighborhood $\O_p$ of $p$ in $\C^B$ such that the  mapping $ p \mapsto \Phi _p $  extends to an analytic mapping from $\O_p$  into $\G_{\g (p)}$. \end{prop}
\begin{proof} We first extend $A_j, j =1, \cdots, N$ analytically on a neighborhood $\O_p$ by Proposition \ref{cont2}. Set $S = S^{2m-1} = \{ f; f \in \C_+^m, \|f\| = 1 \}.$ For each $A_j (q), j = 1, \cdots N, q \in \O_p$ and each $f \in S $ such that $A_j(q) f \not = 0 $, we define again $T_j (q)f $ by:
$$ T_j (q) f \; = \; \frac{A_j(q)f}{\|A_j(q)f\|}.$$           For $p \in \M(B)$, the function  $\Phi _p$ is given by the  limit from formula (\ref{Derriennic}):
$$\Phi _p(\xi ) =  -\ln \lim _{k \to \infty } \frac{ \< \a_1(\xi), T_{j_1}(\xi )\cdots T_{j_{k-1}}(\xi )  f_0 \> } { \< \a(\xi), T_{j_1}(\xi )\cdots T_{j_{k-1}}(\xi )  f_0 \> },$$
where  $f_0 \in S$  the column vector $\{ 1/\sqrt {|B|}, \cdots, 1/\sqrt{|B|}\}$: we use the fact that the limit of $ T_{j_1}(\xi )\cdots T_{j_{k-1}}(\xi )  f $ does not depend on the initial point $f$. 

We have to show that this limit extends  on some neighborhood $\O_p$ of $p$ to an analytic function into $\G_\g$.  Set $$\Phi _{p,k}(\xi )\;  := \; -\ln  \frac{ \< \a_1(\xi), A_{j_1}(\xi )\cdots A_{j_{k-1}}(\xi )  f_0  \> } { \< \a(\xi), A_{j_1}(\xi )\cdots A_{j_{k-1}}(\xi )  f_0 \> }.$$   We are going to find $\O_p$ and $k_0$ such that, for $k \geq k_0$, the functions $\Phi _{p,k} (\xi)$ extend to analytic functions from $\O_p$ into $\G_{\g}$ and, as $k \to \infty $, the functions $\Phi _{p,k} (\xi)$  converge in $\G_{\g}$ uniformly on $\O_p$.

The functions $q \mapsto  \< \a_1(\xi), A_{j_1}(\xi )\cdots A_{j_{k-1}}(\xi )  f_0  \> , q \mapsto \< \a(\xi), A_{j_1}(\xi )\cdots A_{j_{k-1}}(\xi )  f_0 \> $ are polynomials in $q$ and depend only on a finite number of coordinates of $\xi $. Therefore, if we can find a neighborhood $\O_p$ and a $k$ such that  these two functions do not  vanish, then   $\Phi_{p,k} $ extends to an analytic function from $\O_p$ to $\G_\g$.

\

{\it Step1:  Contraction }

By   (\ref{diam1}), (\ref{diam2}) and Proposition \ref{cont2},  we can choose a neighborhood $\O_p$ such that for $q \in \O_p$, the diameter $\D$ of $A_j(q) \C_+^m $ is smaller than $4 (\tanh)^{-1}(\g^{2r})$ for all $j = 1, \cdots , N.$\footnote{One can also use directly \cite{Du2}, Theorem 4.5.} The set $\DD := S \cap \big( \cup _{j} A_j(p) \C_+^m\big)$ is compactly contained in the interior of $S$. We choose a smaller neighborhood $\O_p$  such that, if $q \in \O_p$, $$ \D < 4 (\tanh)^{-1}(\g^{2r}) \quad {\textrm {and}} \quad 0 \not \in A_j (\DD \cup \{f_0\})  \; {\textrm { for }} j = 1, \cdots, N.$$
For $q \in \O_p$, the projective images $T_{j_1}(\xi )\cdots T_{j_{k-1}}(\xi )  f_0 $ are all defined and we have, by repeated application of (\ref{ccontr})
$$\vt_\CC \left(T_{j_1}(\xi )\cdots T_{j_{k-1}}(\xi )  f_0 , T_{j_1}(\xi )\cdots T_{j_{k-1}}(\xi )  f _{k,k'}(\xi )\right) \; \leq \; \g^{2(k-1)r} \vt (f_0,f _{k,k'}(\xi)),$$
where $k'>k$ and $f_{k,k'} (\xi) :=  T_{j_k}(\xi )\cdots T_{j_{k'-1}}(\xi )  f_0.$   The $f_{k,k'} (\xi) $ are all in $\DD$. Then,  $\vt _\CC (f_0, f_{k,k'}(\xi) ) \leq C$ for all $\xi \in \partial F$, all $k, k' \geq 1$. Set 
 $$g = T_{j_1}(\xi )\cdots T_{j_{k-1}}(\xi )  f_0 , \quad g' = T_{j_1}(\xi )\cdots T_{j_{k-1}}(\xi )  f _{k,k'}(\xi ).$$   For all $\xi \in \partial F$, all $k, k' \geq 1$, consider the number $\rho (\xi, k , k')$ associated by Proposition \ref{cdist} to $g$ and $g'$. We have, by (\ref{compar}):
 $$ |\rho(\xi,k,k') | = 1 \quad {\textrm {and}} \quad \| \rho (\xi, k, k') g -g'\| \; \leq \; K C \g^{2kr}.$$
 Since $\a (p,\xi)$ and $\a_1 (p, \xi)$ take finite many values, it follows that
 \begin{eqnarray*}&& |\< \a(p,\xi), g\> \<\a_1(p,\xi), g'\> - \< \a(p,\xi),g'\> \<\a_1(p,\xi),g\>| \\ &=& |\< \a(p,\xi), \rho(\xi,k,k') g\> \<\a_1(p,\xi) ,g'\> - \< \a(p,\xi),g'\> \<\a_1(p,\xi),\rho(\xi,k,k')g\> |\\&\leq &\;  K C \g^{2kr}. \end{eqnarray*}
 Since $g$ and $g'$ are in the compact set $\DD \cup\{f_0\}$, we can, by Proposition \ref{cont}, choose a neighborhood $\O_p$ such that, for all $q \in \O_p $, all $\xi \in \partial F$, all $k < k'$
 \begin{eqnarray*}\label{close}
 && \big|\< \a(\xi),  T_{j_1}(\xi )\cdots T_{j_{k-1}}(\xi )  f_0\> \<\a_1(\xi),  T_{j_1}(\xi )\cdots T_{j_{k'-1}}(\xi )  f_0 \> \\ &\quad&- \< \a(\xi), T_{j_1}(\xi )\cdots T_{j_{k'-1}}(\xi )  f_0\> \<\a_1(\xi), T_{j_1}(\xi )\cdots T_{j_{k-1}}(\xi )  f_0\>\big| \; \leq \;  K C \g^{2kr}.
\end{eqnarray*}

\

{\it Step 2:  The $\Phi _{q,k}$ extend}

Recall that $D $ is the set of unit vectors in the positive quadrant. For $g,g' \in \cup _jT_j(p)(D) \cup \{f_0\}$, $\< \a(p,\xi), g\> \<\a_1(p,\xi), g'\> $ is real positive and bounded away from 0 uniformly in $\xi, g $ and $g'$. Recall the isometric inclusion $\pi : D \to S $ of Proposition \ref{ccone}. There is a neighborhood $\CC _0 $ of $\pi \big(\cup_j T_j(p)(D) \cup \{f_0\}\big)$ in $S$  and $\d >0$ such that for $g,g' \in \CC _0$, 
$\big| \< \a(p,\xi), g\> \<\a_1(p,\xi), g'\>\big|  > \d .$
Of course, we can take $\CC _0$ invariant by multiplication by all $z$ with $|z | = 1$. Then, there exists $\e >0 $ such that if $\vt _{C_+^m} \big(g, \pi (\cup_j T_j(p)(D) \cup \{f_0\} ) \big)<\e, \vt _{C_+^m} \big(g', \pi (\cup_j T_j(p)(D) \cup \{f_0\} ) \big)<\e,$ then $\big| \< \a(p,\xi), g\> \<\a_1(p,\xi), g'\>\big|  > \d/2. $

For $q \in \O_p$ and $k_0 > 1 + \ln(\e /2) / 2r \ln \g$, the $\vt _{\C_+^m}$-diameter of each one of the sets $T_{j_1}(q,\xi )\cdots T_{j_{k_0-1}}(q,\xi ) S$ is smaller than $\e /2$, for all $\xi$. As $\xi $ varies, there is only a finite number of mappings $T_{j_1}(q,\xi )\cdots T_{j_{k_0-1}}(q,\xi ) $. By continuity of $q \mapsto T_j$ (where the  $T_j$s now are considered as  mappings from $\CC /\sim$ into itself), there is a neighborhood $\O_p$ such that for $q \in \O_p$, the Hausdorff distance between $T_{j_1}( q,\xi )\cdots T_{j_{k_0-1}}(q,\xi ) S/ \sim$ and 
$T_{j_1}( p,\xi )\cdots T_{j_{k_0-1}}(p,\xi ) S/ \sim$ is smaller than $\e/2$. It follows that if $q \in \O_p$, and $g,g'$ are in the same $T_{j_1}( q,\xi )\cdots T_{j_{k_0-1}}(q,\xi ) S$ for some $\xi$, then
$$\big| \< \a(p,\xi), g\> \<\a_1(p,\xi), g'\>\big|  > \d/2. $$
By taking a possibly smaller $\O_p$, we have that if $q \in \O_p$, and $g,g'$ are in the same $T_{j_1}( q,\xi )\cdots T_{j_{k_0-1}}(q,\xi ) S$ for some $\xi$, then
$$\big| \< \a(q,\xi), g\> \<\a_1(q,\xi), g'\>\big|  > \d/4. $$
In particular this last  expression does not vanish and $\Phi _{q,k}$ is an analytic function on $\O_p$ for $k \geq k_0$.

\

{\it Step3: The $\Phi _{q,k} $ converge uniformly on $\partial F$} 

Take a neighborhood $\O_p$ and $k_0$ such that for $q \in \O_p$ the conclusions of Steps 1 and 2 hold. We claim that for all $\e >0 $, there is $k_1$  such that for $k,k' \geq k_1$, $q \in \O_p$, $\max _\xi |\Phi _{q,k}(\xi ) - \Phi _{q,k'}(\xi )| < \e $. Suppose $k_1 > k_0$. We have to estimate $$ \max _{\xi} \Big| \ln \frac{ \< \a(\xi), T_{j_1}(\xi )\cdots T_{j_{k-1}}(\xi )  f_0  \> }{ \< \a(\xi), T_{j_1}(\xi )\cdots T_{j_{k'-1}}(\xi )  f_0 \> }
\frac{ \< \a_1(\xi ), T_{j_1}(\xi )\cdots T_{j_{k'-1}}(\xi )  f_0\> }{ \< \a_1(\xi), T_{j_1}(\xi )\cdots T_{j_{k-1}}(\xi )  f_0 \> }\Big|.$$
By the conclusions of Steps 1 and 2, this quantity is smaller that $C\max \{\g^{2kr}, \g^{2k'r} \}$. This is smaller than $\e$ if $k_1$ is large enough.

 \

{\it Step 4: The $\Phi _{q,k} $ converge in norm $\|.\|_{\g(p)}$} 

 With the same $\O_p, k_0$, we now claim that for all $\e >0 $, there is $k_2 = \max \{k_0,  \ln \g / r\ln \e \}$ such that for $k,k' \geq k_2$ and $q \in \O_p$, $ \|\Phi _{q,k}(\xi ) - \Phi _{q,k'}(\xi )\|_\g < \e $.  Let $\xi, \xi' $ be two points of $\partial F$ with $\d(\xi , \xi ' ) \leq \exp (-((n+1) r+1))$. We want to show that there is a constant $C$ independent on $n$, such that, for all $ q\in  \O_p$, all $k,k' \geq k_2$:
$$ |\Phi _{q,k}(\xi) - \Phi _{q,k'}(\xi) - \Phi _{q,k}(\xi') + \Phi _{q,k'} (\xi ') | \; \leq \; C \g ^{(n+1)r +1}\e.$$
Since $k,k' \geq k_0$, the difference $\Phi _{q,k} (\xi) - \Phi _{q,k'} (\xi ) $ is given by:
$$ \Phi _{q,k} (\xi) - \Phi _{q,k'} (\xi ) =  \ln  \frac {\< \a_1, T_{j_1}\cdots T_{j_{k'-1}}  f_0 \>} {\< \a_1, T_{j_1}\cdots T_{j_{k-1}}  f_0 \>} \frac  {\< \a, T_{j_1}\cdots T_{j_{k-1}}  f _0\>}  {\< \a, T_{j_1}\cdots T_{j_{k'-1}} f_0 \>}. $$

For $k,k' \leq n+1$, $ \Phi _{q,k} (\xi) - \Phi _{q,k'} (\xi ) =  \Phi _{q,k} (\xi') - \Phi _{q,k'} (\xi' ) $, and there is nothing to prove. 

Assume $k' >k \geq n+1 $, Step 3 shows that both  $|\Phi _{q,k}(\xi) - \Phi _{q,k'}(\xi)|$ and $|\Phi _{q,k}(\xi') - \Phi _{q,k'} (\xi ')|$ are smaller than $C \g^{2k r } \leq C\g^{nr} \g ^{kr}\leq  C\g^{nr}\e$.

The remaining case, when $k_0 \leq  k \leq n+1 \leq k'$, clearly follows from the other two and this shows Step 4.

\

 Finally we have that  the functions $\Phi _{p,k} $ are analytic and converge uniformly in $\G_\g $ on a neighborhood $\O _p $ of $p$.
The limit is an analytic function on $ \O_p$.

\end{proof} 

\section{Proof of Theorem \ref{main}}
In this section, we consider $\partial F$ as a subshift of finite type and let $\tau $ be the shift transformation on $\partial F$: 
$$ \tau \xi = \eta _1 \eta _2 \cdots  \; \; {\textrm { with }} \; \; \eta _n = \xi _{n+1}.$$
For $\g < 1$ and $\phi \in \G_\g$ with real values, we define the transfer operator $\LL _\phi $ on $\G_\g $ by 
$$ \LL _\phi \psi (\xi ) \; := \; \sum _{\eta \in \tau^{-1}\xi } e^{\phi (\eta ) } \psi (\eta) . $$
Then, $\LL _\phi $ is a bounded operator in $\G_\g $ and, by Ruelle's transfer operator theorem (see e.g. \cite{Bo}), there exists a number $P(\phi ) $, a  positive function $h_\phi \in \G_\g$ and an unique linear functional  $\nu _\phi $ on $\G_\g$ such that:
$$ \LL _\phi  h_\phi = e^{P(\phi )} h_\phi,\quad \LL^\ast _\phi  \nu_\phi = e^{P(\phi )} \nu_\phi  \quad {\textrm{and}} \quad \nu _\phi (1) = 1. $$
The functional $\nu _\phi $ extends to probability measure on $\partial F$ and is the only eigenvector of $\LL^\ast_\phi $ with that property.
Moreover, $\phi \mapsto \LL_\phi $ is a real  analytic  map from $\G_\g $ to the space of linear operators on $\G_\g$ (\cite{R}, page 91). Consequently,   the mapping $\phi \mapsto \nu_\phi $ is real analytic  from $\G_\g$ into the dual space $\G_\g^\ast $ (see e.g. \cite{C}, Corollary 4.6). By Proposition \ref{Lip}, the mapping $p \mapsto \nu _{\Phi _p} $ is real analytic  from a neighborhood of $p$ in $\M(B)$   into the space $\G_{\g(p)}^\ast $.

The main observation is that, for all $p \in \M(B)$, $\LL^\ast_{\Phi _p}p^\infty = p^\infty $; this implies that  $P(\Phi _p ) = 0 $  and that the distribution  $\nu _{\Phi _p}$ is the restriction of  the measure $p^\infty $ to any $\G_\g$ such that $\Phi _p \in \G_\g$. Indeed, we have:
$$ \frac{ d\tau_\ast p^\infty }{ dp^\infty } (\xi ) = \frac {d(\xi _1)_\ast p^\infty }{dp^\infty } = K_{\xi } (\xi _1) = e^{\Phi _p(\xi )}$$
so that, for all continuous $\psi$:
$$ \int (\LL _{\Phi _p} \psi ) dp^\infty  =  \sum _{a} \int _{a\xi , \xi_1 \not = a^{-1}} \frac {dp^\infty (a\xi ) } { dp^\infty (\xi )} \psi (a\xi) dp^\infty (\xi ) = \int \psi dp^\infty. $$
Recall the equations (\ref{entropy}) and (\ref{drift}) for $h_p $ and $\ell _p$. $\ell _p $ is given by a finite sum  (in $x$) of integrals with respect to  $p^\infty $ of the  functions $\xi \mapsto \th _\xi (x) $. Since these functions only depend on a finite number of coordinates in $\partial F$, they belong to $\G_\g$ for all $\g < 1$.  Since $p \mapsto \nu _{\Phi _p} $ is real analytic  from a neighborhood of $p$ into $\G_{\g(p)}^\ast$,  $p\mapsto \ell _p $ is real analytic on a neighborhood of $p$. Since this is true for all $p \in \M(B)$,  the function $p\mapsto \ell _p $ is real analytic on $\M(B)$. 

The argument is the same for $h_p$, since the function $\ln  \frac {dx^{-1}_\ast p^\infty }{dp^\infty } (\xi) = \ln K_\xi ( x^{-1})  \in \G_\g $ for all $x$ and for all $\g, \b <\g <1$ and the mappings $p \mapsto \ln  K_\xi ( x^{-1}) $ are real analytic  from a neighborhood of $p$  into $\G_{\g(p)} $. Indeed, $ \ln K_\xi (\xi _1) \in \G_\b $ by Proposition \ref{holder} and $p \mapsto \ln K_\xi (\xi _1 ) $ is real analytic  into $\G_{\g(p)} $ by Proposition \ref{Lip}. Moreover, 
if $a$ is a generator different from  $\xi _1 $, $ \ln K_\xi (a) = -  \ln K_{a^{-1}\xi } (a^{-1})$ also lies in $\G_\b$ and $p \mapsto \ln K_\xi (a) $  is also real analytic  into $\G_{\g(p)} $. For a general $x \in F$, $x= a_1 \cdots a_t$, write 
$$ K_\xi (x^{-1}) \; = \; K_\xi ( a_t^{-1} \cdots a_1^{-1}) \; = \; K_\xi ( a_t^{-1})  K_{a_t \xi } (a_{t-1}^{-1}) \cdots  K_{a_2\cdots a_t\xi }(a_1^{-1}).$$

\

This completes the proof of Theorem \ref{main}. For the proof of Theorem \ref{expo}, fix an origin $o \in \H^k$. Then, $\pi (F)o $ accumulates to the boundary of $\H^k$ in a Cantor set $\L$ called the limit set of $\pi (F)$. The mapping $\pi _o : F \to \H^n, \pi _o(x) = x.o$ extends to a H\"older continuous mapping $\pi _o$ from $\partial F$ to the limit set $\L$ of $\pi (F)$. We can express the exponent $\g_p$ as:
$$
\g_p \; = \; \lim_n \frac{1}{2n} \sum _{x\in F} d(o, \pi _o(x)) p^{(n)}(x),
$$ 
where the distance $d$ is the hyperbolic distance in $\H^k$. We obtain, in the same way as for formula (\ref{drift}),
\begin{eqnarray*}
\g_p \; &=& \; \frac{1}{2} \sum _{x \in F} \left( \int _{\L}\Theta _\zeta (\pi _o(x^{-1})) d((\pi _o)_\ast p^\infty) (\zeta)\right) p(x) \\ &=& \; \frac{1}{2} \sum _{x \in F} \left( \int _{\partial F}\Theta _{\pi _o(\xi)} (\pi _o(x^{-1})) d(p^\infty) (\xi)\right) p(x),
\end{eqnarray*}
where $\Theta _\zeta$ is now the Busemann function of $\H^k$: $\Theta _\zeta (z) := \lim _{w \to \zeta} d(w,z) - d(w,o).$  Since, for all $x \in F$, the  function $\xi \mapsto \Theta_{\pi _o(\xi)} ( \pi _o(x))$ is a $\rho$-H\"older continuous function for some fixed $\rho$, we deduce as above that $p \mapsto \g_p$ is real analytic on $\M(B)$.


\small\begin{thebibliography}{99}
\bibitem[{\bf Bi}]{B} G. Birkhoff, Extension of Jentzsch's theorem, {\em Trans. Amer. Math. Soc.,} {\bf 85} (1957), 219--227.
\bibitem[{\bf Bo}]{Bo} R. Bowen, Equilibrium states and the ergodic theory of Anosov Diffeomorphisms, {\em Lecture Notes in Math.,} {\bf 470} (1974), Springer, ed.
\bibitem[{\bf Co}]{C} G. Contreras, Regularity of topological and metric entropy of hyperbolic flows, {\em Math. Z.,} {\bf 210} (1992), 97--111.
\bibitem[{\bf De1}]{De1} Y. Derriennic, Marche al\'eatoire sur le groupe libre et fronti\`ere de Martin, {\em Z. Wahrscheinlichkeitstheorie verw. Geb.} {\bf 32} (1975), 261--276.
\bibitem[{\bf De2}]{De2} Y. Derriennic, Quelques applications du th\'eor\`eme ergodique sousadditif, {\em Ast\'erisque} {\bf 74} (1980), 183--201.
\bibitem[{\bf DG}]{DG} Y. Derriennic and Y. Guivarc'h, Th\'eor\`eme du renouvellement pour les groupes non-moyennables, {\em C.R. Acad. Sci. Paris, S\'er. A--B}, {\bf 277} (1973) A613--A615.
\bibitem[{\bf Du1}]{Du} L. Dubois, Projective metrics and contraction principles for complex cones, {\em J. London Math. Soc. (2)} {\bf 79} (2009), 719--737.
\bibitem[{\bf Du2}]{Du2} L. Dubois, An explicit Berry-Ess\'een bound for uniformly expanding maps on the interval, {\em Israel J. Math.}, to appear.
\bibitem[{\bf DM}]{DM} E. B. Dynkin and M. B. Malyutov, Random walks on groups with a finite number of generators, {\em Dokl. Akad. Nauk SSSR}, {\bf 137} (1961), 1042--1045.
\bibitem[{\bf Er}]{E} A. Erschler, On continuity of range, entropy and drift for random walks on groups, {\em preprint.}
\bibitem[{\bf EK}]{EK} A. Erschler and V. A. Kaimanovich, Continuity of entropy for random walks on hyperbolic groups, {\em in preparation.}
\bibitem[{\bf Fu}]{F} H. Furstenberg, Random walks and discrete subgroups of Lie groups, {\em Advances Probab. Related Topics,} {\bf 1} (1971), 1--63.
\bibitem[{\bf G1}]{G1} L. A. Gilch, Rate of escape of random walks on free products, {\em J. Aust. Math. Soc.,} {\bf 83(I)} (2007), 31--54. 
\bibitem[{\bf G2}]{G2} L. A. Gilch, Asymptotic entropy of random walks on free products, {\em preprint}.
\bibitem[{\bf H}]{H} H. Hennion, D\'erivabilit\'e du plus grand exposant caract\'eristique des produits de matrices al\'eatoires ind\'ependantes \`a co\"efficients positifs. {\em Ann. Inst. H. Poincar\'e, Probab Statist.,} {\bf 27} (1991), 27--59.
\bibitem[{\bf Ka}]{K} V.A.  Kaimanovich, The Poisson formula for groups with hyperbolic properties, {\em Ann. Math.,} {\bf 152} (2000), 659--692.
\bibitem[{\bf KL}]{KL} A. Karlsson and F. Ledrappier, Drift and entropy for random walks, {\em Pure Appl. Math. Quarterly,} {\bf 3} (2007), 1027--1036.
\bibitem[{\bf Le}]{Le} F. Ledrappier, Some Asymptotic properties of random walks on free groups, {\it in} Topics in Probability and Lie Groups, J.C. Taylor, ed. {\em CRM Proceedings and Lecture Notes,} {\bf 28} (2001), 117--152.
\bibitem[{\bf M}]{M} J. Mairesse, Randomly growing braid on three strands and the manta ray, {\em Ann. Appl. Probab.,} {\bf 17} (2007), 502--536.
\bibitem[{\bf MM}]{MM} J. Mairesse and F. Math\'eus, Random walks on free products of cyclic groups, {\em J. London Math. Soc.,} {\bf 75} (2007), 47--66.
\bibitem[{\bf P}]{P} Y. Peres, Domains of analytic continuation for the top Lyapunov exponent,  {\em  Ann. Inst. H. Poincar\'e, Probab Statist.,} {\bf 28} (1992), 131--148.
\bibitem[{\bf R1}]{R} D. Ruelle, Thermodynamic formalism, Addison-Wesley, Reading, MA (1978).
\bibitem[{\bf R2}]{R2} D. Ruelle, Analyticity properties of the characteristic exponent of random matrix products, {\em Advances in Math.,} {\bf 32} (1979), 69--80.
\bibitem[{\bf R3}]{R3} D. Ruelle, Repellers for real analytic maps, {\em Ergod. Th. \& Dynam. Sys.,} {\bf 2} (1982) 99--107.
\bibitem[{\bf Ru}]{Ru} H.H. Rugh, Cones and gauges in complex spaces: Spectral gaps and complex Perron-Frobenius theory. {\em Ann. Math.}, {\bf 171} (2010), 1707--1752.

\end{thebibliography}
\end{document}